\documentclass[12pt]{l4dc2020} 

\usepackage{amsmath, amsfonts, amssymb, graphicx,epstopdf,color,soul,mathtools,enumitem}
\usepackage{mathrsfs, mathtools,bbm, dsfont}

\usepackage{subcaption}
\usepackage{ifthen}
\usepackage{multirow, multicol}
\usepackage{float}
\usepackage{subeqnarray, array}
\usepackage{enumitem}
\usepackage{xcolor}
\usepackage[font=small]{caption}
\usepackage{tikz}

\usepackage{times}

\title{
Finite-Time Performance of Distributed Two-Time-Scale Stochastic Approximation}

\author{\Name{Thinh T. Doan} \Email{thinhdoan@gatech.edu}
\AND
 \Name{Justin Romberg} \Email{jrom@ece.gatech.edu}\\
 \addr School of Electrical and Computer Engineering\\
Georgia Institute of Technology, GA, 30332, USA}

\newcommand{\Eset}{\mathbb{E}}

\newcommand{\Rset}{\mathbb{R}}

\newcommand{\Dcal}{{\cal D}}

\newcommand{\Gcal}{{\cal G}}

\newcommand{\Kcal}{{\cal K}}

\newcommand{\Xcal}{{\cal X}}

\newcommand{\Abf}{{\bf A}}
\newcommand{\Bbf}{{\bf B}}

\newcommand{\Vbf}{{\bf V}}
\newcommand{\Wbf}{{\bf W}}
\newcommand{\Xbf}{{\bf X}}
\newcommand{\Ybf}{{\bf Y}}

\newcommand{\bbf}{{\bf b}}

\newcommand{\1}{{\mathbf{1}}}

\newcommand{\bbar}{{\bar{b}}}

\newcommand{\xbar}{{\bar{x}}}
\newcommand{\ybar}{{\bar{y}}}

\newcommand{\xibar}{{\bar{\xi}}}
\newcommand{\psibar}{{\bar{\psi}}}

\newcommand{\xhat}{{\hat{x}}}
\newcommand{\yhat}{{\hat{y}}}

\newcommand{\Xhatbf}{{\hat{\Xbf}}}
\newcommand{\Yhatbf}{{\hat{\Ybf}}}

\newtheorem{thm}{Theorem}

\newtheorem{assump}{Assumption}

\begin{document}

\maketitle

\begin{abstract}
Two-time-scale stochastic approximation is a popular iterative method for finding the solution of a system of two equations. Such methods have found broad applications in many areas, especially in machine learning and reinforcement learning. In this paper, we propose a distributed variant of this method over a network of agents, where the agents use two graphs representing their communication at different speeds due to the nature of their two-time-scale updates. Our main contribution is to provide a finite-time analysis for the performance of the proposed method. In particular, we establish an upper bound for the convergence rates of the mean square errors at the agents to zero as a function of the step sizes and the network topology. 
\end{abstract}


\section{Introduction}
Two-time-scale stochastic approximation ({\sf SA}) is a recursive algorithm for finding the solution of a system of two equations \cite{borkar2008}. In this algorithm, one iterate is updated using step sizes that are very small compared to the ones used to update the other iterate. One can view that the update associated with the small step sizes is implemented at a ``slow" time-scale, while the other is executed at a ``fast" time-scale. In this paper, our focus is to consider a distributed variant of this two-time-scale {\sf SA} in the context of multi-agent systems, where a group of agents can communicate at different speeds through two possibly different connected graphs. Our main goal is to study a finite-time analysis for the performance of the proposed method, where we provide an upper bound for its convergence rate as a function of the two step sizes and the two network topology.

Two-time-scale {\sf SA} and its distributed counterpart have received a surge of interests due to their broad applications in many areas, some examples include optimization \cite{WangFL2017,Polyak1987}, distributed optimization on multi-agent systems \cite{DoanMR2018b,DoanBS2017}, power control for wireless networks \cite{LongZLYG2007}, and especially in reinforcement learning \cite{SBbook1998, KondaT2003, Sutton2009a, leeHe2019}. In these applications, it has been observed that using two-time-scale iterations one can achieve a better performance than the one-time-scale counterpart; for example, the iterates may converge faster \cite{Polyak1987}, the algorithm performs better under communication constraints \cite{DoanBS2017,DoanMR2018b}, and the algorithm is more stable under the so-called off-policy in reinforcement learning \cite{Sutton2009a}.

The existing literature has only focused on the convergence of the centralized two-time-scale {\sf SA}. The asymptotic convergence of this two-time-scale {\sf SA} can be achieved by using the {\sf ODE} methods \cite{BorkarM2000}, while its rates of convergence has been studied in \cite{KondaT2004, DalalTSM2018, KarmakarB2018, GuptaSY2019_twoscale,DoanM2019}. In particular, the work in \cite{DalalTSM2018, KarmakarB2018} provides a concentration bound for the finite-time analysis of this method, while an asymptotic rate has been studied in \cite{KondaT2004}. Recently, a finite-time analysis for the performance of the centralized two-time-scale {\sf SA} has been provided in \cite{GuptaSY2019_twoscale} under constant step sizes and in \cite{DoanM2019} under time-varying step sizes. 

\textbf{Main Contribution}. 
In this paper, we propose a distributed variant of the linear two-time-scale {\sf SA} over a multi-agent system. Due to the two-time-scale updates, the agents use two different graphs representing their communication at two different speeds. Our focus is to provide a finite-time analysis for the proposed method. In particular, we provide an upper bound for the rates of the average of the mean square errors at the nodes to zero, as a function of the two step sizes and the two network topology. We show that this method converges at a rate $\mathcal{O}(1/(1-\sigma)^2k^{2/3})$ under some proper choice of the two step sizes, where $\sigma$ represents the slower mixing time of the two communication graphs and $k$ is the number of iterations. Our theoretical results explicitly show the impacts of the two step sizes and network topology on the performance of the proposed algorithm.           





\section{Distributed linear two-time-scale stochastic approximation} \label{sec:alg}
We consider the problem of finding the solution $(x^*,y^*)\in\Rset^{d}\times\Rset^{d}$ of a linear system of equations defined over a network of $N$ nodes. Associated with each node $i$ is a matrix $\Abf$ and a vector $\bbf^{i}$ 
\begin{align*}
\Abf = \left[\begin{array}{cc}
\Abf_{11}     &  \Abf_{12}\\
\Abf_{21}     & \Abf_{22}
\end{array}\right]\in\Rset^{2d\times 2d},\qquad \bbf^{i} = \left[\begin{array}{c}
     b_{1}^{i} \\
     b_{2}^{i} 
\end{array}\right]\in\Rset^{2d}.    
\end{align*}
The goal of the nodes is to cooperatively find the solution $(x^*,y^*)$ of the system of linear equations
\begin{align}
\Abf_{11}x^* + \Abf_{21}y^*  =   \sum_{i=1}^{N}b_{1}^{i}\qquad\text{and}\qquad \Abf_{21}x^* + \Abf_{22}y^* =  \sum_{i=1}^{N}b_{2}^{i}  .\label{notation:xy*}
\end{align}
We are interested in the situation where a central coordinator is absent, therefore, the nodes have to cooperatively solve this problem. In addition, we assume that the matrices $\Abf_{ij}$ and $b_{1}^{i}$, for all $i,j$, are unknown to node $i$ and each node can only have an access to a noisy observation of these matrices and vectors. Therefore, we consider distributed iterative methods for solving this problem. In particular, we are interested in studying the distributed variant of the linear two-time-scale {\sf SA} \cite{KondaT2004,DoanM2019,GuptaSY2019_twoscale,DalalTSM2018}, where each node $i$ maintains an estimate $(x^{i},y^{i})$ of $(x^*,y^*)$ and iteratively updates its estimates as
\begin{align}
x_{k+1}^{i} &= \sum_{j=1}^{N}w_{ij}x_{k}^{j} - \alpha_{k}(\Abf_{11}x_{k}^{i} + \Abf_{12}y_{k}^{i} - b_{1}^{i} + \xi_{k}^{i}) \label{alg:x}\\
y_{k+1}^{i} &= \sum_{j=1}^{N}v_{ij} y_{k}^{j} - \beta_{k}(\Abf_{21}x_{k}^{i} + \Abf_{22}y_{k}^{i} - b_{2}^{i} + \psi_{k}^{i}),\label{alg:y}
\end{align}
where $\beta_{k}\ll\alpha_{k}$ are two different nonnegative step sizes. In addition, $(w_{ij}, v_{ij})$ are the weights that node $i$ assigns for the iterate $(x^{j},y^{j})$ received from node $j$, a neighbor of node $i$. We denote by $\Wbf = [w_{ij}]\in\Rset^{N\times N}$ and $\Vbf = [v_{ij}]\in\Rset^{N\times N}$ the two adjacency matrices  imposed the communication structures between the nodes, that is, nodes $i$ and $j$ can interact with each other if and only if $w_{ij} > 0$ or $v_{ij} > 0$. Note that $\Wbf$ and $\Vbf$ can represent two different graphs, i.e., the nodes can exchange information in different speeds. In addition, $\{\xi_{k}^{i},\psi_{k}^{i}\}$ are the noise sequences corresponding to observations at each node $i$. Here, the goal of the nodes is to obtain $(x^*,y^*)$, i.e., 
\begin{align*}
\lim_{k\rightarrow\infty} x_{k}^{i} = x^*\qquad\text{and}\qquad \lim_{k\rightarrow\infty} y_{k}^{i} = y^*\qquad \text{a.s.},\quad \forall i\in[1,N].
\end{align*}
Here $\beta_{k}$ is much smaller than $\alpha_{k}$, implying that $x_{k}^{i}$ is updated at a faster time scale than $y_{k}^{i}$. Finally, the adjacency matrices $\Wbf,\Vbf$ is used to present different communication speeds between the nodes.

\subsection{Motivating applications}
We are motivated by the wide applications of \eqref{alg:x} and \eqref{alg:y} in many applications, especially the recent interests in multi-agent reinforcement learning \cite{Mathkar2017_GossipRL,DoanMR2019_DTD(0),DoanMR2019_DTD, ZhangYB2019, YangZHB2018, Kar2013_QDLearning,Wai2018_NIPS,DingOthers2019}. One fundamental and important problem in this area is to estimate the total accumulative return rewards of a stationary policy using linear function approximations, which is referred to as the policy evaluation problems. In this context,  two-time-scale algorithms (e.g., gradient temporal difference learning ({\sf GTD})) have been observed to be more stable and perform better compared to the single-time-scale counterpart (e.g., temporal difference learning ({\sf TD})) in the so-called off-policy settings; see for example \cite{Sutton2009a,Sutton2009b}. Motivated by the distributed variant of {\sf TD} studied in \cite{DoanMR2019_DTD(0)}, we consider a distributed version of {\sf GTD} formulated under the forms of \eqref{alg:x} and \eqref{alg:y}. In particular, a team of agents act in a common environment, get rewarded, update their local estimates of the value function, and then communicate with their neighbors. Let $X_{k}$ be the state of environment, $\gamma$ be the discount factor, $\phi(X_{k})$ be the feature vector of state $X_k$, and $R^{i}(\cdot)$ be the local reward return at agent $i$. Given a sequence of samples $\{X_{k}\}$, the updates at the agents can be viewed as a distributed variant of the {\sf GTD} studied in \cite{Sutton2009b} given as 
\begin{align*}
x_{k+1}^{i} &= \sum_{j=1}^{N}w_{ij}x_{k}^{j} - \alpha_{k}\Big(\phi(X_k)^Tx_{k}^{i} + \big[ \phi(X_{k}) - \gamma\phi(X_{k+1}) \big]^Ty_{k}^{i} - R^{i}(X_{k})\Big)\phi(X_k)\\
y_{k+1}^{i} &= \sum_{j=1}^{N}v_{ij} y_{k}^{j} - \beta_{k}\big[\gamma\phi(X_{k+1})-\phi(X_{k})\big]\phi^T(X_k)x_{k}^{i}.
\end{align*}
At each agent $i$, $y_{k}^{i}$ is the main variable used to estimate the optimal solution, while $x_{k}^{i}$ is an additional auxiliary variable. To put these updates in the form of \eqref{alg:x} and \eqref{alg:y}, we introduce the notation 
\begin{align*}
&\Abf_{11}(X_{k}) = \phi(X_{k})\phi^T(X_k),\; \Abf_{12}(X_{k}) = \phi(X_k)\big[\phi(X_{k})-\gamma\phi(X_{k+1})\big]^T,\; b_{1}^{i}(X_{k}) =  R^{i}(X_k)\phi(X_{k})\\
&\Abf_{21}(X_{k}) = [\gamma\phi(X_{k+1})-\phi(X_{k})]\phi^T(X_k),\quad \Abf_{22}(X_{k}) = 0,\quad b_{2}^{i}(X_{k}) = 0.
\end{align*}
In addition, we denote by $\Abf_{\ell u} = \Eset[\Abf_{\ell u}(X_{k})]$ and $b_{\ell}^{i} = \Eset[b_{\ell}^{i}(X_k)]$, for all $\ell,u = 1,2$ and $i\in[1,N]$. One can reformulate the distributed {\sf GTD} above by introducing $\xi_{k}^{i}$ and $\psi_{k}^i$ as
\begin{align*}
&\xi_{k}^{i} = [\Abf_{11}(X_{k}) - \Abf_{11}]x_{k}^{i} + [\Abf_{12}(X_{k})-\Abf_{12}] y_{k}^{i} + b_{1}^{i}(X_k) - b_{1}^{i},\quad\psi_{k}^{i} = [\Abf_{21}(X_{k}) - \Abf_{21}]x_{k}^{i}
\end{align*}
Let $b_{1} = 1/N\sum_{i}b_{1}^{i}$. The goal of the distributed {\sf GTD} is tried to have all $(x_{k}^{i},y_{k}^{i})$ converge to $(x^*,y^*)$, where $x^* = \Abf_{11}^{-1}\left(\Abf_{21}^Ty^* + b_{1}\right)$ and $y^* = \Abf_{12}^{-1}b_{1}$.\\
Another motivating example of using such distributed two-time-scale algorithms \ref{alg:x} and \ref{alg:y} is to solve distributed optimization problems under communication constraints, where another step size in addition to the one associated with the gradients of the functions is introduced to stabilize the algorithm due to the imperfect communication between agents \cite{DoanBS2017,DoanMR2018b}. Finally, the distributed two-time-scale method studied in this paper can be used to solve a convex relaxation of the popular pose graph estimation in robotic networks  \cite{Choudhary2016}. 

\subsection{Assumptions and notation}
We introduce in this section various assumptions, which are necessary to our analysis given later. We first state an assumption on the matrices $\Abf_{ij}$ to guarantee the existence and uniqueness of  $(x^*,y^*)$. 
\begin{assump}\label{assump:matrix_negativity}
The matrices $\Abf_{11}$ and $\Delta =  \Abf_{22} - \Abf_{21}\Abf_{11}^{-1}\Abf_{12}$ are positive definite but not necessarily symmetric, i.e., their eigenvalues are strictly positive.
\end{assump}
One can relax Assumption \ref{assump:matrix_negativity} to cover the case of complex eigenvalues, i.e., $\Abf_{11}$ and  $\Delta$ have eigenvalues with strictly positive real parts. To simplify the notation of our analysis we, however, assume that these matrices are positive definite. An extension of this work to the case of complex eigenvalues is straightforward; see for example \cite{KondaT2004, DoanM2019}.    

\begin{assump}\label{assump:matrix_bounded}
All the matrices $\Abf_{ij}$ and vectors $b_{1}^{i}$ are uniformly bounded, i.e., $\|\Abf_{ij}\|\leq 1$ and there exists a positive constant $R$ such that $\max\{\|b_{1}^{i}\|,\|b_{2}^{i}\|\}\leq R$ for all $i\in[1,N]$.   
\end{assump}

\begin{assump}\label{assump:doub_stoch}
The matrix $\Wbf$, whose $(i,j)$-th entries are $w_{ij}$, is doubly stochastic with positive diagonal, i.e., $\sum_{j=1}^n w_{ij} = \sum_{i=1}^n w_{ij} = 1$. Moreover, the graph $\Gcal_{\Wbf}$ associated with $\Wbf$ is connected, and $w_{ij} > 0$ if and only if $(i, j)$ is an edge of $\Gcal_{\Wbf}$. Similar conditions are assumed for $\Vbf$.  
\end{assump} 

\begin{assump}\label{assump:noise}
The sequence of random variables $(\xi_{k}^{i},\psi_{k}^{i})$, for all $i\in[1,N]$ and $k\geq0$, is independent of each other, with zero mean and uniformly bounded, i.e., there exists a positive constant $C$ s.t. $\max\{\|\xi_{k}^i\|,\|\psi_{k}^{i}\|\}\leq C$ for all $i\in[1,N]$. Moreover, they have common variances given as 
\begin{align}
\begin{aligned}
\Eset[(\xi_{k}^i)^T\xi_{k}^{i}] = \Gamma_{11},\quad\Eset[(\xi_{k}^i)^T\psi_{k}^{i}] = \Gamma_{12} = \Gamma_{21}^T,\quad\Eset[(\psi_{k}^i)^T\psi_{k}^{i}] = \Gamma_{22}.
\end{aligned}
\label{assump:variance}
\end{align}
\end{assump}
Assumption \ref{assump:matrix_bounded} can be guaranteed through a proper scaling step, while Assumption \ref{assump:doub_stoch} is a standard assumption in distributed consensus algorithms \cite{DoanMR2018c}. Finally, we consider the noise model similar to the one in \cite{KondaT2004,DoanM2019}. 

We denote by $\Xbf,\Ybf\in\Rset^{N\times d}$ the matrices whose $i-$th rows are $(x^{i})^{T}$ and $(y^{i})^{T}$ in $\Rset^{1\times d}$, respectively. Then, the matrix forms of Eqs.\ \eqref{alg:x} and \eqref{alg:y} are given as
\begin{align}
\Xbf_{k+1} &= \Wbf\Xbf_{k} -\alpha_{k}\left(\Xbf_{k}\Abf_{11}^T + \Ybf_{k}\Abf_{12}^T - \Bbf_{1} +\Xi_{k}\right)\label{alg:X}\\
\Ybf_{k+1} &= \Vbf\Ybf_{k} -\alpha_{k}\left(\Xbf_{k}\Abf_{21}^T + \Ybf_{k}\Abf_{22}^T - \Bbf_{2} +\Psi_{k}\right),\label{alg:Y}
\end{align}
where $\Bbf_{1}, \Bbf_{2}, \Xi_{k},$ and $\Psi_{k}$ are the matrices whose $i-$th rows are $(b_{1}^{i})^T,(b_{2}^{i})^T,(\xi_{k}^{i})^T$, and $(\psi_{k}^{i})^T$, respectively.  Given a collection of $x^{1},\ldots,x^{N}$, we use $\xbar$ to denote its average, i.e., $\xbar = \frac{1}{N}\sum_{i=1}^{N} x^i$. Thus, since $\Wbf$ and $\Vbf$ are doubly stochastic matrices and by Eqs.\ \eqref{alg:x} and \eqref{alg:y} we have
\begin{align}
\xbar_{k+1} &= \xbar_{k} - \alpha_{k}\left(\Abf_{11}\xbar_{k} + \Abf_{12}\ybar_{k} - \bbar_{1} + \xibar_{k}\right)\label{alg:xbar}\\
\ybar_{k+1} &= \ybar_{k} - \alpha_{k}\left(\Abf_{21}\xbar_{k} + \Abf_{22}\ybar_{k} - \bbar_{2} + \psibar_{k}\right).\label{alg:ybar}
\end{align}


\section{Finite-time bounds of distributed linear two-time-scale {\sf SA}}\label{sec:results}
We present here the convergence rates of the distributed linear two-time-scale {\sf SA}, where we provide an upper bound for the rates of the average of the mean square errors at the nodes to zero. Our result shows that this quantity decays to zero at a rate $\mathcal{O}(1/(k+1)^{2/3})$. In addition, it also depends on the network topology represented by $1-\sigma$, the algebraic network connectivity of two graphs. 

We first introduce a bit more notation. We denote by $\sigma_{\Wbf}$ and $\sigma_{\Vbf}$ the second larges singular values of $\Wbf$ and $\Vbf$, respectively. By Assumption \ref{assump:doub_stoch} we have $\sigma_{\Wbf},\sigma_{\Vbf}\in(0,1)$; see for example \cite{GR2001}. In addition, we denote by $\sigma$, the slower mixing speed of these two graphs
\begin{align}
\sigma \triangleq \max\{\sigma_{\Wbf},\sigma_{\Vbf}\} \in (0,1).\label{notation:sigma}
\end{align}
Let $\delta \in(\sigma,1)$ and denote by $\Kcal^*$ a positive integer such that
\begin{align}
\Kcal^* \geq \left\lceil(\alpha_{0}/(\delta-\sigma))^{3/2}\right\rceil.   \label{notation:K*}
\end{align}
Finally, since $\lim_{k\rightarrow\infty}\sigma^{k}(k+1) = 0$, without loss of generality we assume that $\sigma^{k} \leq \frac{1}{k+1}$. Our main result, the rate of the distributed two-time-scale {\sf SA}, is stated in the following theorem. 

\begin{thm}\label{thm:rates}
Suppose that Assumptions \ref{assump:matrix_negativity}--\ref{assump:noise} hold. Let $\{x_{k}^{i},y_{k}^{i}\}$, for all $i\in[1,N]$, be generated by \eqref{alg:x} and \eqref{alg:y} with $x_{0}^{i} = y^{i}_{0} = 0$. Let $\{\alpha_{k},\beta_{k}\}$ be the sequence of step sizes chosen as
\begin{align}
\alpha_{k} = \frac{\alpha_{0}}{(k+1)^{2/3}},\qquad \beta_{k} = \frac{\beta_{0}}{k+1}\cdot \label{thm_rates:stepsizes}
\end{align} 
Then, there exits constants $\Dcal,\Dcal_{0},\Dcal_{1}$ given in Lemmas \ref{lem:consensus} and \ref{lem:optimal} below such that
\begin{align}
&\frac{1}{N}\sum_{i=1}^{N}\left(\Eset[\|y_{k}^{i}-y^*\|^2] + \frac{\beta_{k}}{\alpha_{k}}\Eset[\|x_{k}^{i}-x^*\|^2]\right)\notag\\
&\leq \frac{16\Dcal^2\beta_{0}\alpha_{0}\ln^2(\Kcal^*)\sigma^{-2\Kcal^*}}{N(1-\sigma)^2(k+1)^{2/3}} + \frac{16\Dcal^2\beta_{0}\alpha_{0}}{N(1-\sigma)^2(k+2)^{5/3}} +  \frac{2\Dcal_{0}}{(k+1)^{2/3}} + \frac{2\Dcal_{1}\ln(k+1)}{k+1}\cdot\label{thm_rates:Ineq}
\end{align}
\end{thm}
More details about the choice of the two step sizes can be found in \cite{DoanM2019}.



\section{Convergence analysis}
We now present the analysis for the results presented in Theorem \ref{thm:rates}. Our analysis is composed of two main steps. We first show that the estimates $x_{k}^{i}$ and $y_{k}^{i}$ converge to their averages $\xbar_{k}$ and $\ybar_{k}$, respectively. We provide an upper bound for the rates of this convergence. This step is done through considering a residual function, which takes into account the coupling between the two step sizes
\begin{align}
V_{k} = \|\Ybf_{k}-\1\ybar_{k}^T\| + \frac{\beta_{k}}{\alpha_{k}} \|\Xbf_{k} - \1\xbar_{k}^{T}\|.\label{def:lyapunov_consensus}   
\end{align}
Second, we study the convergence of $\xbar_{k}$ and $\ybar_{k}$ to the solutions $x^*$ and $y^*$, respectively. One can view the updates of \eqref{alg:xbar} and \eqref{alg:ybar} as a centralized approach for solving \eqref{notation:xy*}. We, therefore, utilize the results in our previous work to have such convergence \cite{DoanM2019}. Due to the space limit, we skip the analysis of the second step and refer interested readers to \cite{DoanM2019} for more details. Our focus here is to provide the analysis for the first step as follows.    

\begin{lemma}\label{lem:consensus}
Suppose that all assumptions and step sizes in Theorem \ref{thm:rates} hold. Denote by $\Dcal$ a constant 
\begin{align}
\Dcal \triangleq \frac{2\sqrt{N}(R+C)(6\alpha_{0}+1)(\Kcal^*)^{1/3}}{1-\delta},\label{notation:D}
\end{align}
where $\Kcal^*$ is defined in \eqref{notation:K*}. Then we obtain for all $k\geq 0$
\begin{align}
&\sum_{i=1}^N \Big(\|y_{i}^{k}-\ybar_{k}\|^2 + \frac{\beta_{k}}{\alpha_{k}}\|x_{i}^{k}-\xbar_{k}\|^2\Big)\leq \frac{8\Dcal^2\beta_{0}\alpha_{0}\ln^2(\Kcal^*)\sigma^{-2\Kcal^*}}{(1-\sigma)^2(k+1)^{2/3}} + \frac{8\Dcal^2\beta_{0}\alpha_{0}}{(1-\sigma)^2(k+2)^{5/3}}\cdot    \label{lem_consensus:Ineq}
\end{align}
\end{lemma}

\begin{proof}
Let $\xhat^{i} = x^{i} - \xbar$ and $\yhat^{i} = y^{i} - \ybar$. Since $\Wbf$ is doubly stochastic Eqs.\ \eqref{alg:x} and \eqref{alg:xbar} gives
\begin{align*}
\xhat_{k+1}^{i}
& = \sum_{j=1}^{N}w_{ij}\xhat_{k}^{j} - \alpha_{k}\Abf_{11}\xhat_{i}^{k}-\alpha_{k}\Abf_{12}\yhat_{k}^{i}+ \alpha_{k}(b_{1}^{i} - \bbar_{1}) -\alpha_{k}(\xi_{k}^{i} - \xibar_{k}),
\end{align*} 
which implies that 
\begin{align}
\Xhatbf_{k+1} &= \Wbf\Xhatbf_{k} -\alpha_{k}\Xhatbf_{k}\Abf_{11}^T - \alpha_{k}\Yhatbf_{k}\Abf_{12}^T +\alpha_{k}(\Bbf_{1} - \1 b_{1}^T) -\alpha_{k}(\Xi_{k} - \1\xibar_{k}^T).\label{lem_consensus:Eq1a}
\end{align}
Using Assumption \ref{assump:doub_stoch} and the Courant-Fisher theorem yields
$\|\Wbf\Xhatbf_{k}\| \leq \sigma_{\Wbf}\|\Xhatbf_{k}\|,   $ \cite{HJ1985}. Thus, taking the $2-$norm on both sides of \eqref{lem_consensus:Eq1a} and using Assumptions \ref{assump:matrix_bounded} and \ref{assump:noise} gives
\begin{align}
\|\Xhatbf_{k+1}\| \leq (\sigma_{\Wbf}+\alpha_{k})\|\Xhatbf_{k}\| + \alpha_{k}\|\Yhatbf_{k}\| + \sqrt{N}(R + C)\alpha_{k}. \label{lem_consensus:Eq1}    
\end{align}
Similarly, using Eqs.\ \eqref{alg:y} and \eqref{alg:ybar} we have 
\begin{align}
\|\Yhatbf_{k+1}\| \leq (\sigma_{\Vbf}+\beta_{k})\|\Yhatbf_{k}\| + \beta_{k}\|\Xhatbf_{k}\| + \sqrt{N}(R + C)\beta_{k}. \label{lem_consensus:Eq2}
\end{align}
By \eqref{notation:sigma}, $\sigma = \max\{\sigma_{\Vbf},\sigma_{\Wbf}\}\in(0,1)$, and by \eqref{notation:K*}, $\sigma + 2\alpha_{k}\leq \delta \in(\sigma, 1), \forall k\geq \Kcal^*.$ Then, adding Eq.\ \eqref{lem_consensus:Eq1} to Eq.\ \eqref{lem_consensus:Eq2} and using $\beta_{k}\ll\alpha_{k}$ yield
{\small
\begin{align*}
&\|\Xhatbf_{k+1}\|+\|\Yhatbf_{k+1}\|\\ 
&\leq (\sigma+2\alpha_{k}) (\|\Xhatbf_{k}\|+\|\Yhatbf_{k}\|) + 2\sqrt{N}(R+C)\alpha_{k}\leq \delta (\|\Xhatbf_{k}\|+\|\Yhatbf_{k}\|) + 2\sqrt{N}(R+C)\alpha_{k}\notag\\
&\leq \delta^{k+1-\Kcal^*}(\|\Xhatbf_{\Kcal^*}\|+\|\Yhatbf_{\Kcal^*}\|) + 2\sqrt{N}(R+C)\sum_{t=\Kcal^*}^{k}\alpha_{k}\delta^{k-t}\leq (\|\Xhatbf_{\Kcal^*}\|+\|\Yhatbf_{\Kcal^*}\|) + \frac{2\sqrt{N}(R+C)\alpha_{0}}{1-\delta}\cdot
\end{align*}}
Similarly, since $x_{0}^{i} = y_{0}^{i} = 0$ we can obtain
\begin{align*}
\|\Xhatbf_{\Kcal^*}\|+\|\Yhatbf_{\Kcal^*}\| &\leq (\sigma+2\alpha_{0})(\|\Xhatbf_{\Kcal^*-1}\|+\|\Yhatbf_{\Kcal^*-1}\|)+ 2\sqrt{N}(R+C)\alpha_{\Kcal^*-1}\notag\\
&\leq 2\sqrt{N}(R+C)\sum_{t=0}^{\Kcal^*-1}\alpha_{t}\leq 6\sqrt{N}(R+C)\alpha_{0}(\Kcal^*)^{1/3},
\end{align*}
where the last inequality is due to
$\sum_{t=0}^{\Kcal^*-1}\alpha_{t} \leq 3\alpha_{0}(\Kcal^*)^{1/3}.$  Thus, the two preceding relations give
\begin{align}
\|\Xhatbf_{\Kcal^*+1}\|+\|\Yhatbf_{\Kcal^*+1}\|  
&\leq \frac{6\sqrt{N}(R+C)\alpha_{0}(\Kcal^*)^{1/3}}{1-\delta}\cdot\label{lem_consensus:Eq2a}
\end{align}
We denote by $\gamma_{k} = \beta_{k}/\alpha_{k}$, a nonnegative and nonincreasing sequence since $\beta_{k}\ll\alpha_{k}$. Moreover, since $\beta_{0}\leq \alpha_{0}$, we have $\gamma_{k}\leq 1$. We next consider the residual function $V$ in \eqref{def:lyapunov_consensus}. Indeed, using Eqs.\ \eqref{lem_consensus:Eq1} and \eqref{lem_consensus:Eq2} and since $\gamma_{k+1}\leq \gamma_{k}\leq 1$ we have
\begin{align*}
V_{k+1} &= \|\Yhatbf_{k+1}\| + \gamma_{k+1}\|\Xhatbf_{k+1}\|\leq \|\Yhatbf_{k+1}\| + \gamma_{k}\|\Xhatbf_{k+1}\|\notag\\
&\leq (\sigma_{\Vbf}+\beta_{k})\|\Yhatbf_{k}\| + \beta_{k}\|\Xhatbf_{k}\| + 2\sqrt{N}(R + C)\beta_{k} + \sigma_{\Wbf}\gamma_{k}\|\Xhatbf_{k}\| + \beta_{k}\|\Xhatbf_{k}\| + \beta_{k}\|\Yhatbf_{k}\|\notag\\
&\leq \sigma V_{k} + 2\beta_{k}(\|\Yhatbf_{k}\| + \|\Xhatbf_{k}\|) + 2\sqrt{N}(R + C)\beta_{k},
\end{align*}
which by using Eq.\ \eqref{lem_consensus:Eq2a} and $\Dcal$ in \eqref{notation:D} we have for all $k\geq \Kcal^*$
{\small
\begin{align*}
V_{k+1} &\leq \sigma V_{k} + \Dcal\beta_{k}\leq \sigma^{k+1-\Kcal^*}V_{\Kcal^*} + \Dcal\!\!\sum_{t=\Kcal^*}^{k}\beta_{t}\sigma^{k-t}\leq \sigma^{k+1-\Kcal^*}V_{\Kcal^*} + \Dcal\!\!\sum_{t=\Kcal^*}^{\left\lfloor k/2 \right\rfloor}\beta_{t}\sigma^{k-t}+\Dcal\!\!\!\sum_{t=\left\lceil k/2 \right\rceil}^{k}\beta_{t}\sigma^{k-t}\notag\\
&\leq \sigma^{k+1-\Kcal^*}V_{\Kcal^*} + \frac{\Dcal\beta_{0}\sigma^{\lceil k/2\rceil}}{1-\sigma} + \frac{\Dcal\beta_{k/2}}{1-\sigma}\leq \sigma^{k+1-\Kcal^*}V_{\Kcal^*} + \frac{\Dcal\beta_{0}}{1-\sigma}\sigma^{\lceil k/2\rceil} + \frac{2\Dcal\beta_{0}}{(1-\sigma)}\frac{1}{k+1}\cdot
\end{align*}}
Moreover, since $x_{0}^{i} = y_{0}^{i} = 0$ implying $V_{0} = 0$, we have
\begin{align*}
V_{\Kcal^*} &\leq \sigma V_{\Kcal^*-1} + \Dcal\beta_{\Kcal^*-1}\leq \Dcal\beta_{0}\sum_{t=0}^{\Kcal^*-1}\frac{1}{t+1}\leq \Dcal\beta_{0}\ln(\Kcal^*).
\end{align*}
Combining these two relations immediately gives 
\begin{align*}
V_{k+1} &\leq \Dcal\beta_{0}\ln(\Kcal^*)\sigma^{k+1-\Kcal^*} + \frac{\Dcal\beta_{0}}{1-\sigma}\sigma^{\lceil k/2\rceil} + \frac{2\Dcal\beta_{0}}{(1-\sigma)}\frac{1}{k+1}\notag\\ 
&\leq \frac{2\Dcal\beta_{0}\ln(\Kcal^*)\sigma^{-\Kcal^*}}{1-\sigma}\sigma^{\lceil k/2\rceil} + \frac{2\Dcal\beta_{0}}{(1-\sigma)}\frac{1}{k+1}\cdot    
\end{align*}
Using the preceding relation, the definition of $V$ in \eqref{def:lyapunov_consensus}, and $(x+y)^2 \leq 2x^2 + 2y^2$ we obtain
\begin{align*}
&\sum_{i=1}^N \|y_{i}^{k}-\ybar_{k}\|^2 \leq \frac{4\Dcal^2\beta_{0}^2\ln^2(\Kcal^*)\sigma^{-2\Kcal^*}}{(1-\sigma)^2}\sigma^{k} + \frac{4\Dcal^2\beta_{0}^2}{(1-\sigma)^2(k+2)^2}.
\end{align*}
Similarly, we obtain
\begin{align*}
\frac{\beta_{k}}{\alpha_{k}}\sum_{i=1}^N \|x_{i}^{k}-\xbar_{k}\|^2 &\leq \frac{4\Dcal^2\beta_{0}^2\ln^2(\Kcal^*)\sigma^{-2\Kcal^*}}{(1-\sigma)^2}\frac{\sigma^{k}\alpha_{k}}{\beta_{k}} + \frac{4\Dcal^2\beta_{0}^2}{(1-\sigma)^2}\frac{\alpha_{k}}{(k+2)^2\beta_{k}}\notag\\
&\leq \frac{4\Dcal^2\beta_{0}\alpha_{0}\ln^2(\Kcal^*)\sigma^{-2\Kcal^*}}{(1-\sigma)^2}\frac{1}{(k+1)^{2/3}} + \frac{4\Dcal^2\beta_{0}\alpha_{0}}{(1-\sigma)^2}\frac{1}{(k+2)^{5/3}},
\end{align*}
where recall that we assume   $\sigma^{k} \leq 1/(k+1)$. Adding the preceding two relations give Eq.\ \eqref{lem_consensus:Ineq}.   
\end{proof}
We next utilize the following result about the convergence of of $(\xbar_{k},\ybar_{k})$ to the solutions $(x^*,y^*)$.
\begin{lemma}[Theorem $1$ in  \cite{DoanM2019}]\label{lem:optimal}
Suppose that all assumptions and step sizes in Theorem \ref{thm:rates} hold. Then there exists two absolute constants $\Dcal_{0}$ and $\Dcal_{1}$ such that 
\begin{align}
\Eset[\|\ybar_{k}-y^*\|^2] + \frac{\beta_{k}}{\alpha_{k}}\Eset[\|\xbar_{k}-x^*\|^2] \leq \frac{\Dcal_{0}}{(k+1)^{2/3}} + \frac{\Dcal_{1}\ln(k+1)}{k+1}\cdot\label{lem_optimal:Ineq}    
\end{align}
\end{lemma}
Using the results in Lemmas \ref{lem:consensus} and \ref{lem:optimal}, we immediately have the proof of Theorem \ref{thm:rates} as follows. By using Eqs.\ \eqref{lem_consensus:Ineq} and \eqref{lem_optimal:Ineq} we obtain Eq.\ \eqref{thm_rates:Ineq}, i.e.,
\begin{align*}
&\frac{1}{N}\sum_{i=1}^{N}\big(\Eset[\|y_{k}^{i}-y^*\|^2] + \frac{\beta_{k}}{\alpha_{k}}\Eset[\|x_{k}^{i}-x^*\|^2]\big) \notag\\
&\leq \frac{2}{N}\sum_{i=1}^{N}\big(\Eset[\|y_{k}^{i}-\ybar_{k}\|^2 + \Eset[\|\ybar_{k}-y^*\|^2] + \frac{\beta_{k}}{\alpha_{k}}\Eset[\|x_{k}^{i}-\xbar_{k}\|^2] + \frac{\beta_{k}}{\alpha_{k}}\Eset[\|\xbar_{k}-x^*\|^2]] \big)\notag\\
&\leq \frac{16\Dcal^2\beta_{0}\alpha_{0}\ln^2(\Kcal^*)\sigma^{-2\Kcal^*}}{N(1-\sigma)^2(k+1)^{2/3}} + \frac{16\Dcal^2\beta_{0}\alpha_{0}}{N(1-\sigma)^2(k+2)^{5/3}} +  \frac{2\Dcal_{0}}{(k+1)^{2/3}} + \frac{2\Dcal_{1}\ln(k+1)}{k+1}\cdot
\end{align*}
\begin{remark}
Note that the analysis studied in this paper can be extended to cover the case when each node $i$ knows a different matrix $\Abf^{i}$, i.e., associated with each node $i$ is a matrix $\Abf^{i}$ and a vector $\bbf^{i}$ 
\begin{align*}
\Abf^{i} = \left[\begin{array}{cc}
\Abf_{11}^{i}     &  \Abf_{12}^{i}\\
\Abf_{21}^{i}     & \Abf_{22}^{i}
\end{array}\right]\in\Rset^{2d\times 2d},\qquad \bbf^{i} = \left[\begin{array}{c}
     b_{1}^{i} \\
     b_{2}^{i} 
\end{array}\right]\in\Rset^{2d}.
\end{align*}
The goal of the nodes is to cooperatively find the solution $(x^*,y^*)$ of the linear equations
\begin{align}
\begin{aligned}
\sum_{i=1}^{N}\Abf_{11}^{i}x^* + \Abf_{21}^{i} y^*  -   b_{1}^{i} = 0,\quad\text{and}\quad
\sum_{i=1}^{N}\Abf_{21}^{i}x^* + \Abf_{22}^{i}y^* - b_{2}^{i} = 0  .\label{notation:xy*}
\end{aligned}    
\end{align}
However, an additional projection step to a compact set $\Xcal$ containing $(x^*,y^*)$ is needed in this case
\begin{align*}
x_{k+1}^{i} &= \left[\sum_{j=1}^{N}w_{ij}x_{k}^{j} - \alpha_{k}(\Abf_{11}^{i}x_{k}^{i} + \Abf_{12}^{i}y_{k}^{i} - b_{1}^{i} + \xi_{k}^{i})\right]_{\Xcal}\\
y_{k+1}^{i} &= \left[\sum_{j=1}^{N}v_{ij} y_{k}^{j} - \beta_{k}(\Abf_{21}x_{k}^{i} + \Abf_{22}^{i}y_{k}^{i} - b_{2}^{i} + \psi_{k}^{i})\right]_{\Xcal}.
\end{align*}  
Such a projection step is often used in the context of distributed optimization \cite{DoanBS2017,DoanMR2018b}. However, this step may not be practical in reinforcement learning since we may not have reasonable knowledge to decide $\Xcal$.
\end{remark}

\section{Concluding Remarks}
In this paper, we propose a distributed variant of the two-time-scale {\sf SA} for finding the root of a system of two linear equations. Our main contribution is to provide a finite-time analysis of the proposed method, where we show that this method converges at a rate $\mathcal{O}(1/k^{2/3})$. Some other interesting problems including the rates of  nonlinear counterparts \cite{ChenZDMC2019}, impacts of Markovian noise \cite{GuptaSY2019_twoscale}, and applications to study the finite-time analysis  of multi-agent reinforcement learning algorithms could be our future studies.   
\newpage
\section*{Acknowledgment}
This work was supported in part by ARL DCIST CRA W911NF-17-2-018, the Semiconductor Research Corporation (SRC), and DARPA.

\bibliographystyle{IEEEtran}
\bibliography{refs}

\end{document}